\documentclass{article}

\usepackage{amsmath,amsfonts,amsthm,amssymb,amscd,color,xcolor,mathrsfs,verbatim,microtype}
\usepackage{graphicx,eurosym}
\usepackage{hyperref}
\usepackage{mathtools}

\usepackage[applemac]{inputenc}

\usepackage[cyr]{aeguill}

\colorlet{darkblue}{blue!50!black}

\hypersetup{
    colorlinks,%
    citecolor=blue,%
    filecolor=red,%
    linkcolor=darkblue,%
    urlcolor=blue,%
    pdfnewwindow=true,%
    pdfstartview={FitH}
}

\usepackage{graphicx,amscd,mathrsfs,wrapfig,mathrsfs,lipsum}
\usepackage{eufrak}
\usepackage{float}
\usepackage{tikz}
\usepackage{multicol}
\usepackage{caption}
\usetikzlibrary{arrows}
\usepackage{capt-of}

\colorlet{darkblue}{blue!50!black}

\renewcommand{\Re}{\mathop{\rm Re}\nolimits}

\newcommand{\p}{\partial}
\newcommand{\e}{\varepsilon}

\newcommand{\C}{{\mathbb C}}

\newcommand{\R}{{\mathbb R}}
\newcommand{\Z}{{\mathbb Z}}
\newcommand{\IP}{{\mathbb P}}
\newcommand{\pP}{{\mathbb P}}
\newcommand{\I}{{\mathbb I}}
\newcommand{\E}{{\mathbb E}}
\newcommand{\T}{{\mathbb T}}

 \newcommand{\scH}{{\mathscr H}}

\newcommand{\N}{{\mathbb N}}
\newcommand{\la}{\lambda}

\newcommand{\La}{\Lambda}
\newcommand{\ty}{\infty}

\newcommand{\de}{\delta}

\newcommand{\Ker}{\mathop{\rm Ker}\nolimits}

\newcommand{\aA}{{\cal A}}
\newcommand{\BB}{{\cal B}}

\newcommand{\GG}{{\cal G}}
\newcommand{\HH}{{\cal H}}
\newcommand{\II}{{\cal I}}

\newcommand{\KK}{{\cal K}}

\newcommand{\OO}{{\cal O}}
\newcommand{\PP}{{\cal P}}

\newcommand{\RR}{{\cal R}}

\newcommand{\XX}{{\cal X}}

\newcommand{\lag}{\langle}
\newcommand{\rag}{\rangle}

\newcommand{\dd}{{\textup d}}

\newcommand{\PPPP}{{\mathfrak P}}

\let\emptyset\varnothing

\newcommand{\lspan}{\mathop{\rm span}\nolimits}

\newcommand{\supp}{\mathop{\rm supp}\nolimits}

\theoremstyle{plain}

\newtheorem*{mta}{Theorem~\hypertarget{A}{A}}
\newtheorem*{mtb}{Theorem~\hypertarget{B}{B}}

\newtheorem*{lemma*}{Lemma}
\newtheorem{theorem}{Theorem}[section]
\newtheorem{lemma}[theorem]{Lemma}
\newtheorem{proposition}[theorem]{Proposition}

\theoremstyle{definition}
\newtheorem{definition}[theorem]{Definition}

\theoremstyle{remark}

\numberwithin{equation}{section}

\begin{document} 
\author{Vahagn~Nersesyan\,\footnote{Universit\'e Paris-Saclay, UVSQ, CNRS, Laboratoire de Math\'ematiques de Versailles, 78000, Versailles, France;  e-mail: \href{mailto:vahagn.nersesyan@uvsq.fr}{Vahagn.Nersesyan@uvsq.fr}}}
 \date{\today}

\title{The complex Ginzburg--Landau equation  perturbed by a force     localised both in   physical and Fourier~spaces}
\date{\today}
\maketitle

\begin{abstract}

In the paper~\cite{KNS-2018},~a criterion for     exponential mixing  is established    for a class of    random dynamical systems.~In~that~paper, the~criterion is applied to PDEs perturbed by a noise  localised  in the  Fourier space. In~the present paper, we show that, in~the case of the complex Ginzburg--Landau~(CGL) equation, that criterion can be used to consider even more degenerate noise that is localised both in   physical and Fourier~spaces. This~is achieved by
  checking   that the linearised equation is almost surely   approximately  controllable.~We~also study the~problem of    controllability   of the    nonlinear~CGL equation.~Using Agrachev--Sarychev type arguments,   we prove an approximate controllability property in the case of a   control  force which is again   localised     in   physical and Fourier spaces.

   .



\end{abstract}

  \tableofcontents

\setcounter{section}{-1}

\section{Introduction}
\label{S:0}

In this paper, we consider the    complex 
 Ginzburg--Landau (CGL) equation on the torus~$\T^3=\R^3/2\pi\Z^3$ driven    by a very degenerate   random   or control forces.~The main novelty of this paper is the assumption    that the force acts     on an arbitrary non-empty  open set~$\omega\subset \T^3$ trough 
   only few  Fourier modes. More~precisely, we consider the~problem
 \begin{align} 
	\p_tu-(\nu+i)\Delta u+\gamma u+ic|u|^4u&=\chi(x) \, \eta(t,x), \quad x\in \T^3,\label{0.1}\\
	u(0,x)&=u_0(x),\label{0.2}
\end{align}where $\nu, \gamma,c>0$ are some parameters, $\chi:\T^3\to \R_+ $ is a   smooth function such that $\supp \chi \subset \omega$, and     $u=u(t,x)$ is a complex-valued unknown function. Let~us begin with the case where $\eta$ is a random force.~To~simplify the presentation, we~assume in this Introduction that 
  $\eta$ is a   Haar random    process of the   form
\begin{equation}\label{0.3}
	\eta(t,x)=    \sum_{l\in\KK}   \left(b_c^l\eta_c^l(t)\cos\lag l,x \rag    +b_s^l\eta_s^l(t)\sin\lag l,x \rag\right),
\end{equation}
where   $\KK\subset \Z^3$ is the set  
\begin{equation}\label{0.4}
\KK=\{(0,0,0),\,(1,0,0),\, (0,1,0),\,(0,0,1)\},
\end{equation} $  \{b_c^l, b_s^l\}_{l\in \KK}$ are positive   numbers, and  $\eta_c^l=\eta_{c,1}^l+i\eta_{c,2}^l$ and $\eta_s^l=\eta_{s,1}^l+i\eta_{s,2}^l$ are complex-valued  processes 
with $\{\eta_{c,j}^l,    \eta_{s,j}^l:l\in \KK, j=1,2\}$ being   independent copies of a real-valued  process    given~by
$$
	\tilde\eta(t)=\sum_{k=0}^\infty \xi_k h_0(t-k)
	+\sum_{j=1}^\infty \sum_{m=0}^\infty j^{-q}\xi_{jm}h_{jm}(t).
$$
Here~$\{\xi_k,\xi_{jm}\}$  are independent identically distributed (i.i.d.) random variables with Lipschitz-continuous density~$\rho$,     $\{ h_0,h_{jm}\}$ is the Haar basis  (see~Section~5.2 in~\cite{KNS-2018}), and $q>1$.~The~restriction   to integer times   of the  solution  of the problem~\eqref{0.1},~\eqref{0.2}  defines  a family of Markov processes~$(u_k, \IP_u)$ parametrised by the initial condition  $u_0= u \in H^1(\T^3,\C)$.
 Let $\La\subset \T^3$ be the level set for the maximum of the  funciton $\chi$, i.e., 
\begin{equation}\label{0.5}
 \La=\{x\in \T^3:\chi (x)=M\}, \quad \text{where  $M=\max_{x\in \T^3} \chi (x)$.}
\end{equation}
 We prove the following~result.
 \begin{mta}Assume that the   set $ \La$ has a nonempty interior,
  the support of the density~$\rho$ is bounded, and     $\rho(0)>0$.~Then
   the process~$(u_k, \pP_u)$ has a
unique stationary measure  on $H^1(\T^3,\C)$  which is 
  exponentially mixing
in the dual-Lipschitz metric.
 \end{mta}
See Section~\ref{S:1} for more general version of this theorem.~The ergodicity of    randomly forced PDEs has been extensively studied in the literature, mainly in the case of non-degenerate noises.~We refer the reader  to         the~papers~\cite{FM-1995,KS-cmp2000,EMS-2001,BKL-2002} for the       first  results and  the book    \cite{KS-book} and the  reviews~\cite{Flandoli-08, debussche-2013}    for further   references and
     discussions of different          methods.

      We prove Theorem~A by    using a criterion for   ergodicity established in the recent paper~\cite{KNS-2018}. According to that result,    exponential mixing     holds if the resolving operator of the equation has    suitable regularity   properties, admits one globally stable equilibrium,   and has an   almost surely non-degenerate   derivative (see~Conditions~\hyperlink{H1}{(H$_1$)}-\hyperlink{H3}{(H$_3$)} in Section~\ref{S:1.1}).~In~\cite{KNS-2018}, that    criterion is applied   in the case  of the Navier--Stokes~(NS) system and the CGL equation  driven by a noise of the form~\eqref{0.3} acting on all the domain (i.e., when~$ \chi\equiv1$ on~$\T^3$ in the case of Eq.~\eqref{0.1}). The~main difficulty of the problem  considered in the present paper is the verification of the non-degeneracy  property  for the derivative.~It  is related to  the approximate controllability of the linearised equation and is checked      combining trigonometric lie-algebraic computations and   a unique continuation property for linear parabolic equations. Let us stress  that   we   use in an essential way the  local nature of the nonlinear term in the CGL equation, and the   problem remains open in the case of the NS  system.

     The controllability   approach used in this paper
  has been developed starting with the papers~\cite{shirikyan-asens2015, shirikyan-2018}, where     exponential mixing is established  for the NS~system with   a~space-time or  boundary   localised    noise.~In~\cite{KNS-2019}, a~version of the 
        criterion of~\cite{KNS-2018} is derived, where the condition of  existence of a globally stable equilibrium is replaced by a weaker property of approximate controllability for the nonlinear equation.~The criterion of~\cite{KNS-2018} is applied in~\cite{BGN-2020} to   the system of 3D primitive equations of meteorology and oceanology   with a noise only in the temperature equation, and in~\cite{vnersesyan-2019}, to~the   NS system in unbounded domains.~In~\cite{JNPS-2019}, the~controllability approach is further developed to establish a   Donsker--Varadhan type large deviations principle for the Lagrangian trajectories of the NS system.

In the case of the Burgers equation driven by a white-in-time  noise   localised in   physical and Fourier~spaces, the   uniqueness of  stationary measure and mixing   follow    from the approach of~\cite{Borit-13}, although it is not explicitly stated there.    For~the same equation, but with a forcing   that is a sum of an arbitrary smooth deterministic 
function   and a two-dimensional noise     localised in   any subinterval, the   mixing is obtained in~\cite{Shi-18} using a controllability property to trajectories. The proofs of both   papers use in an essential way the   strong dissipative
character of the Burgers equation and do not work in the case of the CGL equation.
  Let~us also recall that,   in~the case of the NS system with a white-in-time noise that is degenerate-in-Fourier (but not in physical) space,    the Malliavin calculus has been used in the papers~\cite{HM-2006, HM-2011} to prove   exponential mixing for    the   NS system.~A similar result is obtained in~\cite{FGRT-2015} in   the case of the Boussinesq~system.        
       
 In the second part of this paper,    we study the problem of approximate controllability of the nonlinear CGL equation with a control   localised both in physical and Fourier~spaces.~Because of some well-known obstructions, 
 the approximate controllability property does not hold in the entire phase space (e.g.,~see~\cite{DR-96, H-1978}).~However, using     Agrachev--Sarychev type arguments, we show  that the restriction of the trajectory to the interior $\OO$ of the level set~$\La$  is approximately controllable  to any target. More precisely, we prove the following~result.
  \begin{mtb}Let us consider the vector space   
\begin{equation}\label{0.6} 
 \HH(\KK)=\lspan\{\cos\lag l,x\rag,\, \sin\lag l,x\rag:l\in\KK\},
\end{equation}where $\KK\subset \Z^3$ is the set given  by~\eqref{0.4}.~The CGL equation is approximately controllable on the set~$\OO$ in small time
by~$\HH(\KK)$-valued control~$\eta$, i.e., for any $\e>0$, any  $T_0>0$, and any~$u_0, u_1 \in   L^2(\T^3,\C)$,    there is a time $T\in (0,T_0)$,  a~control $\eta
\in  L^2([0,T];\HH(\KK)) $, and a unique solution~$u$ of the problem \eqref{0.1},~\eqref{0.2} defined on the interval $[0,T]$   such that
$$
\|u(T)  -   u_0-u_1 \I_{\OO}\|_{L^2}<\e,
$$    where $\I_{\OO}$ is the indicator function of the set $\OO$.
 \end{mtb} 
 In other words, this theorem allows to control approximately the   trajectory on $\OO$      while keeping it close to the          initial condition on   $\T^3\setminus \OO$.
 The   local nature  of the nonlinearity in  the CGL equation is again important for the arguments, and the   problem  is open in the case of the NS system;  see~the 7$^\text{th}$~problem formulated by Agrachev in~\cite{Agrachev-2014}.~In Section~\ref{S:3}, we prove different extensions of Theorem~B   in a   more general setting.~In particular, we~consider the equation in arbitrary space dimension  and the     degree of the nonlinearity is arbitrary,  so~the equation is not necessarily globally well-posed.~We~refer the reader to that section for more details  and a short literature review on control problems with finite-dimensioal~forces.

  This paper is organised as follows.~In Section~\ref{S:1}, we briefly recall the formulation of the  abstract criterion  of~\cite{KNS-2018} and apply it to establish exponential mixing for the CGL equation.~In~Section~\ref{S:2}, we verify the non-degeneracy condition by showing that the~linearised CGL equation is    almost surely approximately controllable.~Section~\ref{S:3} is devoted to the  study  of   approximate controllability of  the nonlinear  CGL equation.~Finally, in Section~\ref{S:4}, we give examples of   saturating~spaces for both linear and nonlinear control problems.

 \subsubsection*{Acknowledgement}

  The author thanks Armen Shirikyan for helpful discussions.
 This research   was supported by the ANR   grant NONSTOPS
ANR-17-CE40-0006-02.

 \subsubsection*{Notation}

    Let $X$ be a   Polish space, that is, a complete separable metric space.~We denote by~$d_X$ the metric on $X$ and by  $B_X(u, R)$ the closed ball of radius $R>0$ centred at~$u\in X$.~The  Borel $\sigma$-algebra on $X$ is denoted by  ${\cal B}(X)$ and the set of Borel probability measures    by   $\PP(X)$.~We use the following spaces, metrics, and norms.

   \smallskip
\noindent
$C_b(X)$    is the space of continuous functions $f:X\to \C$ endowed with the norm~$\|f\|_{\ty}=\sup_{u\in X}|f(u)|$.~We write $C(X)$, when $X$ is compact.

  \smallskip
\noindent
$L_b(X)$   is the space of     functions $f\in C_b(X)$    such that
$$
 \|f\|_{L(X)}=\|f\|_\infty+\sup_{u\neq v} \frac{
{|f(u)-f(v)|}}{d_X(u,v)}<\ty.
$$ 
The dual-Lipschitz
metric on
   $\PP(X)$ is defined by    
\begin{equation}\label{0.7}
  \|\mu_1-\mu_2\|_{L(X)}^*=\sup_{\|f\|_{L(X)}\le 1} |\lag f,\mu_1\rag-\lag f,\mu_2\rag|, \quad \mu_1, \mu_2\in \PP(X),
\end{equation}
  where  $\lag f,\mu\rag=\int_X f(u)\,\mu(\dd u)$.~Now, assume that  $X$ is a Banach  space endowed with a
norm $\|\cdot\|_X$ and let $J_T=[0, T ]$.    

  \smallskip
\noindent 
$L^p(J_T;X)$,   $1\leq p<\infty$  is the
space of measurable functions $f: J_T \rightarrow X$ such~that
$$
\|f\|_{L^p(J_T;X)}=\left(\int_0^T \|f(t)\|_X^p\dd t
\right)^{\frac1p}<\ty.
$$ 

\noindent 
$L^p_{\text{loc}}(\R_+;X)$,  $1\leq p<\infty$ is the space of measurable functions $f: \R_+ \rightarrow X$ such that $f|_{J_T} \in L^p(J_T;X)$ for any  $T>0$.

 \smallskip
\noindent 
 $C(J_T;X)$ is the space of continuous functions $f:J_T\to X$ endowed with the~norm 
$$
\|f\|_{C(J_T;X)}=\sup_{t\in J_T} \|f(t)\|_X.
$$

 \noindent
 $L^2=L^2(\T^d; \C)$  and $H^s=H^s(\T^d; \C),$ $s\ge0$    are the   usual Lebesgue and Sobolev    spaces of functions  $f:\T^d\to  \C$.
  We   consider~$L^2$ as a real Hilbert space with the scalar product and the norm
$$
(u,v)_{L^2}=\Re\int_{\T^d}u(x) \bar v(x)\,\dd x, 	\quad\quad \|u\|_{L^2}=\sqrt{(u,u)_{L^2}}
$$
and endow the   spaces~$H^s$ with the corresponding  scalar products $(\cdot,\cdot)_{H^s}$ and    norms~$\|\cdot\|_{H^s}$.~Throughout this paper,   $C$ denotes   unessential positive constants that may change from line to line.

        \section{Exponential mixing}\label{S:1}

 We start this section by  recalling the formulation of the  
     abstract criterion   established in~\cite{KNS-2018}. Then we explain how it is  applied to prove   exponential mixing for the CGL equation with localised noise.
     
  \subsection{Abstract criterion}\label{S:1.1}

In this subsection, we consider  a random dynamical~system of the form   
\begin{equation} \label{E:1.1}
u_k=S(u_{k-1},\eta_k), \quad k\ge1,
\end{equation}
where $S:H\times E\to H$ is a continuous mapping,~$H$ and~$E$ are real separable  Hilbert spaces, and~$\{\eta_k\}$ are   i.i.d. random variables in~$E$.~We assume   that the law~$\ell$ of~$\eta_k$ has a compact support in~$E$, denoted by $\KK$, and there is a compact set~$X\subset H$ such that $S(X\times\KK)\subset X$. We consider the Markov process~$(u_k,\IP_u)$ obtained by restricting the system~\eqref{E:1.1} to the set~$X$. The~associated   Markov operators are denoted by~$\PPPP_k:C(X)\to C(X)$ and~$\PPPP_k^*:\PP(X)\to\PP(X)$.~Recall that $\mu \in \PP(X)$ is    a  stationary measure  if~$\PPPP^*_1\mu=\mu$.   We assume that the following conditions are satisfied for $S$ and $\{\eta_k\}$.

\smallskip
\begin{description}
\item[\hypertarget{H1}{(H$_1$)}] There   is   a  Banach space $ V $ that is compactly embedded into~$H$ such that
the mapping $S:H\times E\to V$ is twice continuously differentiable  and its~derivatives      are bounded on bounded subsets of~$H\times E$.~Furthermore,
 the mapping~$\eta\mapsto S(u,\eta)$ is analytic from~$E$ to~$H$ for any fixed~$u\in H$,    and the derivatives~$(D_\eta^jS)(u,\eta)$ are continuous in~$(u,\eta)$ and   bounded on bounded subsets of~$H\times E$.

\item[\hypertarget{H2}{(H$_2$)}] 
  There is  constant $a\in(0,1)$ and elements  $\hat u\in X$ and $\hat \eta\in\KK$   such that 
\begin{equation} \label{E:1.2}
\| S ( u, \hat \eta) - \hat u\|_H\le a \,\|u-\hat u\|_H,  \quad u\in X. 
\end{equation} 
\end{description}
 For any   $u\in X$, let~$\KK^u$ be the set of  elements $\eta\in\KK$ such that the image of the mapping $(D_\eta S)(u,\eta):E\to H$ is dense in~$H$. Then $\KK^u\in \BB(E)$.
\begin{description}
\item[\hypertarget{H3}{(H$_3$)}] 
  We have $\ell(\KK^u)=1$ for any $u\in X$. 

\item[\hypertarget{H4}{(H$_4$)}]
  The random variables $\eta_k$ are of the~form
\begin{equation} \label{E:1.3}
\eta_k=\sum_{j=1}^\infty b_j\xi_{jk}e_j, \quad k\ge1,
\end{equation}
where~$b_j>0$ are such that 
$\sum_{j=1}^\infty b_j^2<\infty$, $\{e_j\}$ is an orthonormal basis in the space $E$, and    
 $\xi_{jk}$ are independent real-valued random variables such that~$|\xi_{jk}|\le1$ almost surely.
Moreover, the law of $\xi_{jk}$ has a   Lipschitz-continuous density~$\rho_j$ with respect to the Lebesgue measure. 
\end{description}
    The following  is Theorem~1.1 in~\cite{KNS-2018}. 
\begin{theorem} \label{T:1.1}
Under the   Conditions~{\rm\hyperlink{H1}{(H$_1$)}--\hyperlink{H4}{(H$_4$)}},  the   process $(u_k,\IP_u)$ has a unique stationary measure $\mu\in \PP(X)$. Moreover, $\mu$ is exponentially mixing in the   sense that  there are numbers~$\sigma>0$ and $C>0$   such that
\begin{equation} \label{EEER1.6}
\|\PPPP_k^*\lambda-\mu\|_{L(X)}^*\le Ce^{-\sigma k},  \quad \lambda\in\PP(X),\, k\ge 1, 
\end{equation}where $\|\cdot\|_{L(X)}^*$ is the dual-Lipschitz metric  defined by~\eqref{0.7}.
\end{theorem}
Usually, in   applications  the mapping $S$ is the resolving operator of a parabolic PDE (see \cite{KNS-2018, KNS-2019,  vnersesyan-2019, BGN-2020}).~Conditions~{\rm\hyperlink{H1}{(H$_1$)} and~{\rm\hyperlink{H2}{(H$_2$)} are  standard     regularity  and dissipativity properties  satisfied for a large class of     equations. The    non-degeneracy Condition~{\rm\hyperlink{H3}{(H$_3$)} is less standard; as~we~will see in Section~\ref{S:2},  it~can be verified using   some control theory arguments.

    \subsection{Formulation and proof}\label{S:1.2}

  Let us turn to the   CGL equation~\eqref{0.1}. In this section,  we   assume   that       
  $\eta$     is a    random process    of  the~form
$$
\eta(t,x)=\sum_{k=1}^\infty \I_{[k-1,k)}(t)\eta_k(t-k+1,x),
$$
where   $\I_{[k-1,k)}$ is the indicator function of the interval $[k-1,k)$, and~$\{\eta_k\}$ are i.i.d.   bounded  random variables in $L^2(J;H^2)$ with $J=[0,1]$. Let 
$$
S:H^1\times L^2(J;H^2)\to H^1, \quad (u_0, \eta_1)\mapsto u(1)
$$ be the time-1 resolving operator of the problem~\eqref{0.1},~\eqref{0.2}, and let~$(u_k,\IP_u)$ be the Markov process defined  by~\eqref{E:1.1}. Then the following functional
  $$
  \scH(u)=\int_{\T^3}\left( \frac12 |\nabla u(x)|^2+ \frac{c}{6}| u(x)|^6\right) \dd x, \quad u\in H^1
  $$  is a 
  Lyapunov functional for   the process~$(u_k,\IP_u)$ in the sense that there are numbers $c\in (0,1)$ and $C>0$ such that
  \begin{equation}\label{1.5}
  \E_u\scH(u_k)	\le c^k \left(1+\scH(u)\right)+  C, \quad u\in H^1,\,\, k\ge1, 
  \end{equation}where $\E_u$ is the expectation with respect to~$\pP_u$.~Inequality~\eqref{1.5} is obtained  from estimate~(7.15) in~\cite{KNS-2018}  by taking the expectation.     For any $\la\in \PP(H^1)$, we~set
\begin{gather*}
  \scH(\la)=\int_{H^1} \scH(u)\, \la(\dd u), \\
   \PP_1(H^1)=\left\{\la\in\PP_1(H^1): \scH(\la)<\ty\right\}.
\end{gather*}

  Let~$ \La\subset \T^3$  be  the level set    of  the function~$\chi$  defined by~\eqref{0.5}, and let us assume  that   the interior~$\OO$ of~$ \La$ is non-empty.~We~define a notion of    $\OO$-saturating subspace    as follows.~Let $\HH\subset H^2$ be a finite-dimensional subspace that 
is invariant under complex conjugation, i.e.,~$\bar\zeta\in\HH$ for all~$\zeta\in\HH$, and assume that~$\HH$ contains the function identically equal to~$1$ on~$\T^3$.~Consider a non-decreasing sequence  of finite-dimensional subspaces~$\{\HH_j\}$ of $H^2$ defined as~follows:
\begin{gather}
\HH_0=\HH, \quad \HH_j=\lspan\{\eta,\, \zeta \xi:\, \eta, \zeta\in \HH_{j-1},\, \xi\in \HH\}, \quad j\ge1,\label{1.6}\\\HH_\ty=\cup_{j=0}^\ty \HH_j.\label{1.7}
\end{gather}
\begin{definition}\label{D:1.2}
The subspace
	  $\HH $ is said to be   $\OO$-saturating  if the subspace~$ \HH_\ty $ restricted to the set~$\OO$     is dense in~$L^2(\OO;\C)$.
\end{definition}
Next, we recall the notion  of observable function  introduced in~\cite{KNS-2018}.  
\begin{definition} \label{D:1.3}
Let~$\HH \subset H^2$ be a finite-dimensional subspace,  and let  $\{\varphi_l\}_{l\in\II}$ be an orthonormal basis in~$\HH$ with respect to the scalar product $(\cdot, \cdot)_{H^2}$.	A function~$\zeta\in L^2(J_T;\HH)$, where $J_T=[0,T]$,  is said to be    observable   if for any Lipschitz-continuous functions~$a_l:J_T\to  \R$, $l\in\II$ and any continuous function~$b:J_T\to  \R$,   	the equality 
\begin{equation}\label{1.8}
		\sum_{l\in\II}a_l(t)(\zeta(t),\varphi_l)_{H^2}-b(t)=0\quad\mbox{in $L^2(J_T;\R)$}
\end{equation}
	implies that $a_l= 0 $, $l\in\II$ and~$b= 0$ on $J_T$.  
\end{definition}
It is easy to see that this definition   does not depend on the choice of the basis $\{\varphi_l\}$ in $\HH$.  
 \begin{theorem} \label{T:1.4}
 Let $\chi:\T^3\to \R_+$ be a smooth function such that $\OO\neq \emptyset$, and
	let  $\HH\subset H^2$ be an $\OO$-saturating subspace.~Assume that~$\{\eta_k\}$ are   i.i.d. random variables   in $E=L^2(J;\HH)$ such that the following properties hold. 
\begin{description}
	\item [$\bullet$]  
The law~$\ell\in\PP(E)$ of the random variable~$\eta_k$ has a compact support $\KK$ in $E$ containing the zero. Moreover,  Condition~\hyperlink{H4}{{\rm(H$_4$)}} is satisfied. 
	\item [$\bullet$]
There is $T\in(0,1)$ such that  the restriction of~$\eta_k$ to the interval~$J_T$ is almost surely observable.
\end{description}	
Then, for any  $\nu, \gamma,c>0$, the Markov process~$(u_k,\IP_u)$ has a unique stationary measure~$\mu\in\PP(H^1)$,  and  there are numbers~$\varkappa>0$ and $C>0$ such that
\begin{equation} \label{1.9}
\|\PPPP_k^*\lambda-\mu\|_{L(H^1)}^*\le Ce^{-\varkappa k} \left( 1+\scH(\la)\right), \quad \lambda\in\PP_1(H^1),\, k\ge 1. 
\end{equation}
\end{theorem}
\begin{proof}  By Theorem~7.4 in~\cite{KNS-2018}, the   process~$(u_k,\IP_u)$  possesses a compact invariant absorbing set~$X\subset H^1$ that is closed and bounded in~$H^2$.~Conditions~\hyperlink{H1}{(H$_1$)}--\hyperlink{H4}{(H$_4$)} in Theorem~\ref{T:1.1} are satisfied for the restriction of~$(u_k,\IP_u)$ to~$X$ if we take~$H=H^1$ and~$E=L^2(J;\HH)$.~Indeed, the verification of Conditions~\hyperlink{H1}{(H$_1$)} and~\hyperlink{H2}{(H$_2$)} is carried out in the same way as in the case~$\chi\equiv 1$ considered in Theorem~4.7 in~\cite{KNS-2018}, Condition~\hyperlink{H3}{(H$_3$)} is verified in the next section,     and~\hyperlink{H4}{(H$_4$)} holds by assumption.~Thus, by Theorem~\ref{T:1.1}, we have      exponential mixing~\eqref{EEER1.6}.  Then
	the validity of~\eqref{1.9}   follows from \eqref{1.5} and the      regularisation property  of the CGL equation;    e.g., see Section~3.2.4 in~\cite{KS-book} for a similar argument in the case of the NS system. 
   \end{proof}

   	   For any finite set~$\II\subset\Z^3$ containing the zero vector, let us consider the     subspace 
\begin{equation}\label{1.10}
\HH(\II)=\lspan\{\cos\lag l,x\rag,\, \sin\lag l,x\rag:l\in\II\}.
\end{equation}
Let  $\tilde \II$ be the set   of   all 
      linear combinations of vectors in~$\II$ with integer coefficients. Recall that~$\II$ is a generator if $\tilde \II=\Z^3$.~By Proposition~\ref{P:4.1}, the subspace~$\HH(\II)$ is  $\OO$-saturating   if and only if~$\II$ is a generator.  In particular,   $\HH(\KK)$ is  $\OO$-saturating, where~$\KK\subset \Z^3$ is defined by~\eqref{0.4}.~Furthermore, in   Section~5.2 in~\cite{KNS-2018}, it is proved that  the Haar process~\eqref{0.3}   satisfies the observability condition. Thus,~Theorem~A formulated in the Introduction follows as a consequence of    Theorem~\ref{T:1.4}.

  \section{Controllability of the linearised equation}\label{S:2}

We use the same notation as in the previous section. The objective of this section is to   prove the following~result.
\begin{theorem}\label{T:2.1}
Under the conditions of Theorem~\ref{T:1.4},  	for any $u\in X$ and $\ell$-almost every~$\eta\in E$, the image of the linear mapping $(D_\eta S)(u,\eta):E\to H^1$ is dense in~$H^1$.
\end{theorem}
 \begin{proof}

Let us denote by $\tilde u\in  W^{1,2}(J;H^1)\cap L^2(J;H^3)$ the solution of Eq.~\eqref{0.1} corresponding to the initial   condition~$\tilde u(0)=u\in X$ and consider the  linearised~problem      
	\begin{equation} \label{2.1}
		\dot v+Lv+Q(\tilde u;v)= \chi g, \quad v(0)=0, 
	\end{equation}
	where
	\begin{align}
	 L&=-(\nu+i)\Delta+\gamma,\label{2.2}\\
	 Q(u;v)&=ic\left(3|u|^4v+2|u|^2u^2\bar v\right).\nonumber
	 \end{align}     For  any $g\in E$, let\footnote{To simplify the notation, we do not indicate the dependence of~$A_t$  and other quantities    on~$\tilde u$.} $A_tg$ be the solution of the problem~\eqref{2.1}.~The~theorem will be~proved if we show that   the vector space~$\{A_1 g:g\in E\}$ is dense in~$H^1$ for~$\ell$-a.e.~$\eta\in E$.~Let us fix a realisation of~$\eta$ and a time 
 $T\in (0,1)$ such that   the observability property holds on the interval~$J_T$.~We~are going to show that  the space $\{A_T g: g\in E\}$ is dense in~$L^2$.~This, combined with the parabolic regularisation and a density property of the set  of solutions of linear parabolic equations proved in Proposition~7.2 in~\cite{KNS-2018},   will imply the required property.

 Let  $R(t,s):L^2\to L^2$,  $0\le s\le t\le T$ be   the resolving operator of the  homogeneous   problem
$$		
\dot v+Lv+Q(\tilde u;v)=0, \quad v(s)=v_0. 
$$
  Then, by   the Duhamel formula, we have
	$$
	A_t g= \int_0^t R(t,s) \chi g(s)\,\dd s. 
	$$ 	
	Moreover,  the function
	\begin{equation}\label{2.3}
		w(s)=R^*(T,s)w_0,
	\end{equation}where $R(t,s)^*:L^2\to L^2$ is the adjoint of the operator $R(t,s)$ in $L^2$, 
    is the solution of the  backward problem 
	\begin{equation}\label{2.4}
	  \dot w-L^*w-  Q^*(\tilde u;w)=0, \quad w(T)=w_0  
	\end{equation} with $L^*=-(\nu-i)\Delta+\gamma$ and 
	$$
	Q^*(\tilde u;w)=ic\left(-3|u|^4w+2|u|^2u^2\bar w\right).
	$$
 Let ${\mathsf P}_\HH:L^2\to L^2$ be      the orthogonal projection onto~$\HH$ in~$L^2$. We need to prove that the image of the linear operator
	$$
	\aA:L^2(J_T; L^2)\to L^2, \quad   \aA=A_T{\mathsf P}_\HH
	$$
	is dense in $L^2$.~It suffices to  show that the kernel of the adjoint operator
	$$
	\aA^*: L^2 \to L^2(J_T; L^2), \quad\aA^* ={\mathsf P}_\HH \chi  R(T,s)^*  
		$$ is trivial.~To prove this,  let  us  take any $w_0\in \Ker \aA^*$.~Then~${\mathsf P}_\HH \chi R (T,t)^*w_0=0$ for any $t\in J_T$, hence (cf.~\eqref{2.3}) 
$$
	(\chi \zeta, w(t))_{L^2}=0,\quad  t\in J_T 
$$
for any $\zeta\in\HH_0=\HH$. Now, assuming that we have the relation
\begin{equation} \label{2.5}
	(\chi^p \zeta, w(t))_{L^2}=0,\quad  t\in J_T 
\end{equation}
 for any~$\zeta\in\HH_k$ and some $p\ge1$, let us   prove that
 \begin{equation} \label{2.6}
	(\chi^{p+4} \xi, w(t))_{L^2}=0,\quad  t\in J_T 
\end{equation}
   for any~$\xi\in\HH_{k+1}$.~Indeed,
differentiating~\eqref{2.5} in $t$ and using Eq.~\eqref{2.4}, we obtain 
\begin{equation} \label{2.7}
	\left(L(\chi^p\zeta)+\chi^p Q(\tilde u(t);\zeta),w(t)\right)_{L^2}=0,\quad t\in J_T.
\end{equation}
   Let us put $\eta^l(t)=(\eta(t),\varphi_l)_{H^2}$, $l\in\II$ and write 
\begin{align*}
\eta(t)&=\sum_{l\in\II}\eta^l(t)\varphi_l.
\end{align*}
 Differentiating~\eqref{2.7} in $t$ and using Eqs.~\eqref{0.1} and~\eqref{2.4}, we get
\begin{gather*}
	\left(L(\chi^{p}\zeta)+\chi^{p} Q(\tilde u;\zeta),\dot w\right)_{L^2}  -\left(\chi^{p} B_2(\tilde u;\zeta,L\tilde u+B(\tilde u)),w\right)_{L^2}\\
	  +\sum_{l\in\II}\left(\chi^{p+1} B_2(\tilde u;\zeta,\varphi_l),w\right)_{L^2}\eta^l(t)=0, 
\end{gather*}
where $B_k (u ; \cdot)$ is the $k^\text{th}$ derivative of~$B(u)=ic|u|^4u$ (so $B_1 (u ; \cdot)=Q(u ; \cdot)$ and~$B_5 (\tilde u ; \cdot)=B_5 (\cdot)$ is independent of $\tilde u$).   Thus, we have~\eqref{1.8}, where 
\begin{align*}
	a_l(t)&=\left(\chi^{p+1} B_2(\tilde u(t);\zeta,\varphi_l),w(t)\right)_{L^2},\\
	b(t)&=	\left(L(\chi^p\zeta)+\chi^p Q(\tilde u(t);\zeta),\dot w(t)\right)_{L^2}\\&\qquad\qquad\qquad -\left(\chi^p B_2(\tilde u(t);\zeta,L\tilde u(t)+B(\tilde u(t))),w(t)\right)_{L^2}. 
\end{align*} 
Then the functions $a_l, l\in \II$ are Lipschitz-continuous and $b$ is continuous, so the  observability assumption implies that  
\begin{equation*}
	\left(\chi^{p+1} B_2(\tilde u(t);\zeta,\varphi_l),w(t)\right)_{L^2}=0,
	\quad  l\in\II,\,\, t\in J_T. 
\end{equation*}
Iterating   the same argument three more times, we get 
$$	\left(\chi^{p+4} B_5(\zeta,\varphi_l,\varphi_j,\varphi_m,\varphi_n),w(t)\right)_{L^2}=0,
	\quad  j,l,m,n\in\II,\,\, t\in J_T.
$$Using the equality
 $$
 B_5(\zeta,\xi,1,1,1)
=12ic\,(3\zeta\xi+\bar\zeta\bar\xi +3\bar\zeta\xi+3\zeta\bar\xi)
$$
and the facts that $1\in \HH$ and the spaces~$\HH$ and~$\HH_k$ are invariant under complex conjugation, we arrive at~\eqref{2.6}.~Thus, we    proved the following property: for~any~$\zeta\in  \HH_\ty  $, there is a sequence of integers $p_n\to +\ty $ as $n\to+ \ty$ such that 
$$
	(\chi^{p_n} \zeta, w(t))_{L^2}=0,\quad  n\ge1,\,\,t\in J_T. 
$$Dividing this equality by $M^{p_n}$, where $M=\max_{x\in \T^3} \chi (x)$,   passing to the limit as~$n\to +\ty$, and using    the Lebesgue theorem on dominated
convergence,  we~obtain
$$
	(\zeta, w(t))_{L^2(\OO; \C)}=0,\quad  t\in J_T. 
$$From the saturation property it follows that $w(t,x)=0$ for any $t\in J_T$ and $x\in \OO$. The unique continuation property for parabolic equations (e.g., see~\cite{SS-87}) implies that~$w(t,x)=0$ for any $t\in J_T$ and $x\in \T^3$.~In particular,   $w(T)=w_0=0$ (see~\eqref{2.4}). Thus~$\Ker \aA^*=\{0\}$, which completes the proof of the~theorem. 
\end{proof}

   \section{Controllability of the nonlinear equation}\label{S:3}

\subsection{Preliminaries}\label{S:3.1}
In this section, we consider the problem of   controllability of the following nonlinear CGL equation  on the torus of arbitrary dimension $d\ge1$:
 \begin{equation} 
	\p_tu+ L u  +B(u)=f(t,x), \quad x\in \T^d,\label{3.1}
\end{equation}
  where  $\nu>0$ and\footnote{In this section, we do not assume that the unforced equation admits   one globally stable equilibrium, so the value  $\gamma=0$ is allowed.} $\gamma\ge 0$ are some parameters, $L$ is defined by \eqref{2.2}, and~$B(u)$ denotes the nonlinear term  $ic|u|^{2p}u$ with arbitrary integer  $p\ge1$ and parameter~$c>0$. This equation is supplemented with the initial condition
   \begin{equation} 
	u(0,x)=u_0(x) \label{3.2}
\end{equation}which is assumed to belong to  
   a Sobolev space~$H^s$ of integer order~$s>d/2$, so that the Cauchy problem is
    locally  well-posed  in the sense of      the   following proposition.  For~any $T>0$,   let us introduce  the space 
    $$
    \XX_T= C(J_T;H^s)\cap L^2(J_T;H^{s+1})
    $$   endowed   with the norm 
    $$
    \|u\|_{\XX_T}=\|u\|_{C(J_T;H^s)}+ \|u\|_{ L^2(J_T;H^{s+1})}.
    $$ Together with Eq.~\eqref{3.1}, we   consider   the following more general equation:
\begin{equation}\label{3.3}
	\p_tu+ L (u+\zeta)  + B(u+\zeta)=f(t,x).	
\end{equation}
   \begin{proposition}\label{P:3.1}
For any~$\hat u_0\in H^s$,~$\hat \zeta \in C
(\R_+;H^{s+1})$, and~$\hat f\in L^2_{\textup{loc}}(\R_+;H^{s-1})$,  there is  a time    $T_*:=T_*(\hat u_0, \hat \zeta, \hat f)>0$ and a unique solution $\hat u$ of the problem \eqref{3.3},~\eqref{3.2} with data~$(u_0, \zeta, f)=(\hat u_0, \hat \zeta, \hat f)$     whose restriction to the interval~$J_T$  belongs  to the space~$\XX_T$ for any $T<T_*$.~Furthermore,  there are   
  constants
 $\delta=\delta(T,\lambda)>0$  and $C=C(T,\lambda)>0$, where   
$$
\lambda= \|\hat \zeta\|_{C(J_T;H^{s+1})}+\|\hat f\|_{L^2(J_T;H^{s-1})}+ \|\hat u\|_{\XX_T},
$$
such that  
\begin{enumerate}
\item[$\bullet$] for any $u_0 \in  H^s$, $\zeta\in C(J_T;H^{s+1}),$ and $f\in L^2(J_T;H^{s-1})$  satisfying  
\begin{equation}\label{3.4}
\|u_0 - \hat u_0\|_s +  \| \zeta - \hat \zeta\|_{C(J_T;H^{s+1})} + \|f-\hat f\|_{L^2(J_T;H^{s-1})}<
\delta,
\end{equation}
  the problem~\eqref{3.3},~\eqref{3.2} has a unique
solution $ u\in \XX_T$;
\item[$\bullet$]
  let $\RR$ be the   mapping  taking  $(u_0,\zeta,f )$ satisfying \eqref{3.4} to the solution~$u$.~Then     
 \begin{align*}
\|\RR(u_0,\zeta,f) -\RR(\hat u_0,\hat\zeta,\hat f)\|_{\XX_T}&\le 
C \,\big(\|u_0-\hat u_0\|_s+
\|\zeta-\hat\zeta\|_{C(J_T;H^{s+1})}\nonumber\\&\quad +\|f-\hat f\|_{L^2(J_T;H^{s-1})}\big).
 \end{align*}
 \end{enumerate}
\end{proposition}This proposition is proved by literally repeating the arguments of the proof of Proposition~1 in~\cite{VN-2019A}, where parabolic equation is considered with a real-valued polynomially growing nonlinearity.~In what follows, we      assume that the source term  is of the~form $f=h+\chi\eta$:
\begin{equation}\label{3.5}
\p_tu+ L u  +B(u)=h(t,x)+\chi(x)\eta(t,x),
\end{equation}
 where  $\chi:\T^d\to \R_+$ is a smooth function, $h\in L^2_{\text{loc}}   (\R_+;H^{s-1})$ is a given function, and  $\eta$ is a control   taking  values in a finite-dimensional subspace $\HH \subset H^{s+2}$ that is specified below.~For any~$u_0\in H^s$, $T>0$, and~$\zeta \in C(J_T;H^{s+1})$, let $\Theta(u_0,\zeta,T)$ be the set of controls~$\eta \in  L^2(J_T;H^{s-1})$ such that the problem~\eqref{3.3},~\eqref{3.2} has a unique
solution in~$\XX_T$.~From Proposition~\ref{P:3.1} it follows that the set~$\Theta(u_0,\zeta,T)$ is  open in~$L^2(J_T;H^{s-1})$.~Let  $\RR_t$ be the restriction of the resolving operator~$\RR$   at time~$t\in J_T$.

As it is explained in the references~\cite{DR-96, H-1978}, one cannot control the trajectories of Eq.~\eqref{3.5} outside the support of the function~$\chi$.~We prove that the approximate controllability still holds if we restrict the problem to 
   the interior~$\OO\subset \T^d$    of the level set $\Lambda$ of $\chi$ given by 
\begin{equation}\label{3.6}
 \La=\{x\in \T^d:\chi (x)=M\}, \quad \text{where  $M=\max_{x\in \T^d} \chi (x)$.}
\end{equation}   
More precisely,  we use the following notion of   controllability.
\begin{definition}\label{D:3.2}   
  Eq.~\eqref{3.5} is said to be  approximately controllable on the set~$\OO$  in small time
by~$\HH$-valued control if, for any $\e>0$, any $T_0>0$, any $u_0 \in H^s$, and any~$u_1\in L^2$,    there is a time $T\in (0,T_0)$ and  a control $\eta
\in \Theta(u_0,0, T)\cap L^2(J_T;\HH) $ such that
$$
\|\RR_T(u_0,0,h+\chi\eta) -  u_0-u_1 \I_{\OO}\|_{L^2}<\e,
$$    where $\I_{\OO}$ is the indicator function of the set $\OO$.
\end{definition} 
 Let   $\HH \subset H^{s+2}$  be a finite-dimensional subspace that 
is invariant under complex conjugation    and contains the function identically equal to~$1$ on~$\T^d$. Let~us define a    sequence  of finite-dimensional subspaces~$\{\HH_j'\}$ of $H^{s+2}$ by
\begin{gather}
\HH_0'=\HH, \quad \HH_j'=\lspan\left\{B(\zeta):\, \zeta\in \HH_{j-1}'\right\}, \quad j\ge1,\label{3.7}\\\HH_\ty'=\cup_{j=0}^\ty \HH_j'.\label{3.8}
\end{gather}The proof of the below lemma is postponed to   Section~\ref{S:4.2}.
\begin{lemma}\label{L:3.3} The following equality holds:
\begin{equation}
\HH_j'= \lspan\left\{ \zeta_1\cdot\ldots\cdot\zeta_{2p+1}:\,\,\zeta_l\in\HH_{j-1}',\,\, l=1,\ldots,2p+1 \right\}, \quad  j\ge1.
 \label{3.9}
\end{equation}	
\end{lemma}As a consequence of this lemma, we see that   the sequence  $\{\HH_j'\}$ is non-decreasing.~We use the following notion of saturation in the case of the nonlinear CGL equation. 
  \begin{definition}\label{D:3.4}
 The subspace     $\HH$ is  $\OO$-saturating for Eq.~\eqref{3.5}  
 if the subspace $ \HH_\ty' $  restricted to the set~$\OO$     is dense in~$L^2(\OO;\C)$.
\end{definition} 
 The following is the main result of this section.
   \begin{theorem}\label{T:3.5}  
  Assume that $\chi$  is   such that $\OO\neq \emptyset$,   and
	   $\HH $ is an $\OO$-saturating subspace    in the sense of Definition~\ref{D:3.4}.~Then
   Eq.~\eqref{3.5}  is approximately controllable on $\OO$ in small time   by     $\HH$-valued~control.
\end{theorem}
 By Proposition~\ref{P:4.2},  
the subspace~$\HH(\KK)$  defined by~\eqref{0.4} and~\eqref{0.6} 
 is $\T^3$-saturating in the sense of Definition~\ref{D:3.4}.~Hence, 
 Theorem~B given  in the Introduction is obtained as a particular   case of Theorem~\ref{T:3.5}.

 In the case when  $\chi\equiv1$ on $\T^d$ and under a stronger saturation assumption,   an approximate controllability property of usual form holds for~Eq.~\eqref{3.5}.
 \begin{definition}\label{D:3.6}   
  Eq.~\eqref{3.5} is said to be approximately controllable
by $\HH$-valued control if, for any $\e>0$, any $T>0$, and any $u_0,u_1\in H^s$,   there is  a control $\eta
\in \Theta(u_0,0, T)\cap L^2(J_T;\HH) $ such that
$$
\|\RR_T(u_0,0,h+\eta) - u_1 \|_s<\e.
$$   
\end{definition}  

     \begin{definition}\label{D:3.7}
The subspace   $\HH  $ is saturating for Eq.~\eqref{3.5}  if $\HH_\ty'$  is dense in~$ H^s$.
\end{definition}
    \begin{theorem}\label{T:3.8}  
  Assume that $\chi\equiv1$ on $\T^d$,  and
	   $\HH $ is a saturating subspace in the sense of Definition~\ref{D:3.7}.~Then
   Eq.~\eqref{3.5}  is approximately controllable by $\HH$-valued control in the sense of Definition~\ref{D:3.6}.
\end{theorem}
  A stronger  version of this theorem holds when the subspace~$\HH$ is of a special form.~More precisely, let~$\II\subset\Z^d$ be a finite set containing the zero vector, and let the subspace~$\HH(\II)$ and the set~$\tilde \II\subset \Z^d$ be defined as in the end of Section~\ref{S:1.2}. Furthermore, let $H^s(\II)$ be the closure in $H^s$ of the~subspace~$  \HH(\tilde \II)$.

 \begin{theorem}\label{T:3.9}
Assume that $h\in L^2  (J_T;H^{s-1}(\II))$.~Then Eq.~\eqref{3.5} is approximately controllable by $\HH(\II)$-valued control in the sense of Definition~\ref{D:3.6} if and only if  $\II$ is a generator.
\end{theorem} 
Theorems~\ref{T:3.5}, \ref{T:3.8}, and \ref{T:3.9} are proved in the next subsection.~The approximate controllability of nonlinear heat equations has been studied in the papers~\cite{FPZ-95, FCZ-00} when 
 the     control  is localised in the physical space (but not in   Fourier) and the nonlinear term     grows slowly.~The proof of the above three theorems is inspired by the approach of Agrachev and Sarychev introduced in the  papers~\cite{AS-2005, AS-2006} to consider the     approximate controllability of  the 2D   NS  and Euler systems.~That approach has been further extended and developed   by many authors to   various~PDEs. See the papers~\cite{shirikyan-cmp2006, shirikyan-aihp2007, Ners-2015, VN-2021} for the study of the case of the~3D NS     system,~\cite{SSRodrig-06, RD-2018} for the case of the NS system  on    rectangles with   Lions boundary conditions,~\cite{Hayk-2010,   Hayk-2011} for the 3D Euler system,~\cite{Sar-2012}  for the  2D   cubic Schr\"odinger equation, and~\cite{BGN-2020} for the 3D system of primitive equations of meteorology and oceanology.~The arguments we use in the current setting are closer to the
  ones of the paper~\cite{VN-2019A},  where   parabolic equation  is~considered with a   polynomially growing nonlinearity.

 In all the  above papers, equations with additive controls   are considered.  Let us also mention    the recent paper~\cite{DN-2021}, where   a version of Agrachev--Sarychev technics is proposed to  study the controllability of the  nonlinear Schr\"odinger equation with a multiplicative control.

  \subsection{Proof  of the theorems}\label{S:3.2}

Let $\hat{\Theta}(u_0,T)$ be the   set of pairs 
$$
(\eta, \zeta)\in L^2(J_T;H^{s-1})\times C(J_T;H^{s+1})
$$ such that the problem~\eqref{3.3},~\eqref{3.2} with $f=h+\chi\eta$ has a unique
solution in~$\XX_T$. The following proposition plays an  important role in our arguments. 
  \begin{proposition}\label{P:3.9}  For any  $u_0, \eta\in H^{s+1}$,  $\zeta \in  H^{s+2}$, and   $h\in L^2  (J;H^{s-1})$, there is  $\de_0>0$ such that $(\de^{-1}\eta,\de^{-1/q}\zeta)\in \hat{\Theta}(u_0,\de)$ for any $\de\in (0,\de_0)$, and     the following  limit  holds: 
$$
 	 \RR_{\de}(u_0,\de^{-1/q}\zeta,h+\de^{-1}\chi\eta)\to  u_0+\chi\eta-B(\zeta)  \quad\text{in $H^s$ as $\de\to 0^+$},
$$ where $q=2p+1$.
\end{proposition}This is proved in the same way as Proposition~2 in~\cite{VN-2019A}; we shall not dwell on the details.

  \begin{proof}[Proof of Theorem~\ref{T:3.5}]

To begin with, let us assume that~$u_0\in H^{s+1}$.
  
 {\it Step~1.~Controllability   to  $u_0+ \chi\HH_0$.}~Let us first note that  the problem~\eqref{3.5}, \eqref{3.2} is approximately controllable in small time to any target in the set~$u_0+\chi \HH_0$, i.e.,     
for any $\e>0$, $\eta\in\HH_0$, and $T_0>0$, there are~$T\in (0,T_0)$  and~$\hat \eta
\in \Theta(u_0,T)\cap L^2(J_T;\HH)$ such that  
$$
 \| \RR_T (u_0, 0 , h + \chi \hat \eta) - u_0-\chi \eta \|_s< \e.
$$
Indeed, this follows from   Proposition~\ref{P:3.9} applied for the pair  $(\eta, \zeta)= (\chi\eta,0)$: 
$$ 	
 \RR_{\de}(u_0,0,h+\de^{-1}\chi \eta)\to  u_0+\chi\eta \quad\text{in $H^s$ as $\de\to 0^+$}.
$$
This  implies    the required property with  $T=\delta$ and $\hat\eta=\delta^{-1}\eta$.

{\it Step~2.~Controllability  to  $u_0+\chi^{q^N} \HH_N$.}~Arguing by induction on $N\ge0$, let~us show that the problem~\eqref{3.5}, \eqref{3.2}    is approximately controllable in small time to any target  in~$u_0+\chi^{q^N} \HH_N$.~The base case $N=0$ is considered in   step~1.~Assume that the property is proved for~$N-1$, and let~$\eta\in \HH_N$.~Then, there are vectors~$\zeta_1, \ldots, \zeta_n\in \HH_{N-1}$ such that 
\begin{equation}\label{3.10}
\eta=B(\zeta_1)+\ldots+B(\zeta_n).
\end{equation} Applying Proposition~\ref{P:3.9} for the pair~$(\eta, \zeta)= (0,\chi^{q^{N-1}}\zeta_1)$, we obtain 
\begin{equation}\label{3.11}
 \RR_{\de}(u_0,\de^{-1/q}\chi^{q^{N-1}}\zeta_1,h)\to  u_0-\chi^{q^N}B(\zeta_1)  \quad\text{in $H^s$ as $\de\to 0^+$}.
\end{equation}
 On the other hand,  the following  equality holds
$$
\RR_{t}(u_0+\delta^{-1/q}\chi^{q^{N-1}}\zeta_1,0,h)=\RR_{t}(u_0,\delta^{-1/q}\chi^{q^{N-1}}\zeta_1,h)+\delta^{-1/q}\chi^{q^{N-1}}\zeta_1, \quad t\in J_\delta
$$by the uniqueness of the solution of the Cauchy problem.~Taking  in this equality $t=\delta$ and  using   \eqref{3.11}, we get  
$$
\|\RR_{\delta}(u_0+\delta^{-1/q}\chi^{q^{N-1}}\zeta_1,0,h)-u_0+\chi^{q^{N}}B(\zeta_1) -\delta^{-1/q}\chi^{q^{N-1}}\zeta_1\|_s\to 0\quad \text{as $\delta\to 0^+.$}	
$$
This limit,   the assumption that  $\zeta_1\in \HH_{N-1}$,   the   induction   hypothesis, and Proposition~\ref{P:3.1} imply  that there is a small time $T>0$ and  a   control $  \eta_1
\in \Theta(u_0,0,T)\cap L^2(J_T;\HH)$ such that 
$$ 
	\|\RR_T(u_0,0,h+\chi\eta_1) - u_0+\chi^{q^N}B(\zeta_1) \|_s<\e.
$$
Iterating this argument  for   $\zeta_2, \ldots, \zeta_n$, we construct a small time 
$\hat T >0$ and  a   control $\hat \eta
\in \Theta(u_0,0, \hat T)\cap L^2(J_{\hat T},\HH)$ such that (cf.~\eqref{3.10})
\begin{gather*}
	\|\RR_{\hat T}(u_0,0,h+\chi\hat\eta) -  u_0 +\chi^{q^N}(B(\zeta_1) +\ldots+B(\zeta_n)) \|_s
	\\=\|\RR_{\hat T}(u_0,0,h+\chi\hat\eta) -  u_0+\chi^{q^N}\eta  \|_s<\e.
\end{gather*}
Thus, we have  approximate controllability in small time to any target in  the set~$u_0+\chi^{q^N} \HH_N$.

{\it Step~3.~Conclusion.}~Without loss of generality, we can assume that the maximum~$M$ of the function~$\chi$ (see~\eqref{3.6}) equals to~1. 
 Let  us take any $u_1\in L^2$. By the saturation hypothesis (Definition~\ref{D:3.4}), 
there is an integer $N\ge1$ and a vector~$  \eta\in  \HH_N$ such~that 
\begin{equation}\label{3.12}
\| u_0+\eta-  \hat u_1\|_{L^2(\OO;\C)}<\e/2.
\end{equation}On the other hand, by the fact that the sequence $\{\HH_j\}$ is non-decreasing  and
  the results of steps~1 and~2, for any $\e>0$ and $T_0>0$, there are   sequences of integers~$\{N_n\}\subset \N $, times  
   $\{T_n\}\subset  (0,T_0)$,  and     controls $  \{\eta_n\} \subset 
  \Theta(u_0,0,T_n)\cap L^2(J_{T_n},\HH)$ such that~$N_n\to +\ty $~and 
\begin{equation}\label{3.13}
\|\RR_{T_n}(u_0,0,h+\chi \eta_n) - u_0-\chi^{q^{N_n}}\eta \|_s<\e/2, \quad n\ge1.
\end{equation}
  Combining this with \eqref{3.12}, the fact that $\chi(x)\in [0,1)$ for  $x\in   \T^2\setminus \OO$, and the Lebesgue theorem on dominated
convergence, we derive     approximate controllability   in small time   from the  initial position~$u_0\in H^{s+1}$ to the target~$u_0+\I_{\OO} u_1$ in the~$L^2$-norm.~Taking $\eta$ equal to  zero on a small interval of time and using the   regularising property of the CGL equation, we   obtain   approximate controllability in small time    from    arbitrary~$u_0\in H^s$.  
 \end{proof}

  \begin{proof}[Proof of Theorem~\ref{T:3.8}] The starting point of the proof is~\eqref{3.13}, where we   take  $\chi\equiv1$.~The saturation assumption (Definition~\ref{D:3.7}) implies that~we have 
     approximate controllability in small time in the sense that, for any $\e>0$,~$T_0>0$,   
    and~$u_0,u_1\in H^s$,   there are    $T\in (0,T_0)$ and $\eta
\in \Theta(u_0,0,T)\cap L^2(J_T;\HH) $ such that
$$
\|\RR_T(u_0,0,h+\eta) - u_1 \|_s<\e.
$$  Thus,  the theorem will be proved if we show that,  for any $\e>0$, $T>0$   and~$u_1\in H^s$,    there is  $\eta
\in \Theta(u_1,0,T)\cap L^2(J_T;\HH)  $ verifying 
$$
\|\RR_T(u_1,0,h+\eta) - u_1 \|_s<\e
$$ (with initial condition coinciding with the target~$u_1$).~By Proposition~\ref{P:3.1}, there are constants~$r\in (0,\e)$ and~$\tau>0$   such~that    $(0,0)\in  \hat\Theta(v,\tau)$   and
$$
	\|\RR_t(v,0,h)-u_1\|_s<\e \quad  \text{for any $v \in B_{H^s}(u_1,r), \,\,  t\in J_\tau$}  .
$$   
 If $\tau\ge T$, then   the proof of the theorem is complete.~If~$\tau< T$, we use the approximate controllability property    with  initial condition~$u_0=\RR_\tau(v,h)$, small time $T'< T-\tau$, and target $u_1$. Thus, we find   $\hat \eta\in \Theta(u_0,0,T')\cap L^2(J_{T'},\HH)$ such~that 
$$\RR_{T'}(u_0,0,h+\hat\eta) \in B_{H^s}(u_1,r).
$$   Again,  if $2\tau+T'>T$, then   the proof is complete.  If $2\tau+T'<T$, we apply    the controllability property to return to~$B_{H^s}(u_1,r)$. Iterating finitely many times this argument, we complete the proof.
  \end{proof}

  \begin{proof}[Proof of Theorem~\ref{T:3.9}] 
    If $\II$ is a generator, then $\HH(\II)$ is saturating by   Proposition~\ref{P:4.2}.~Applying Theorem~\ref{T:3.8}, we derive the required  approximate controllability property. 
    
    If $\II$   is not a generator, let $l$ be any vector in the non-empty set      $\Z^d\setminus \tilde \II$. The~assumption that~$h\in L^2  (J_T;H^{s-1}(\II))$, the fact that~$H^s(\II)$ is invariant for the linear CGL equation (i.e., Eq.~\eqref{3.5} with $B=0$), and   that the   term $B$   maps~$H^s(\II)$ to itself imply that       
the~set     $$
   \aA= \{\RR_T(0,0,h+\eta): \eta\in  \Theta(u_0,0,T)\cap L^2(J_T;\HH) \}  
    $$ is contained in  $H^s(\II)$.~Thus,    $\cos\lag l,x\rag$ and $\sin\lag l,x\rag$ are orthogonal to $\aA$, so the set attainable from the origin with $\HH$-valued controls is not dense in~$H^s$.   
    \end{proof}

  \section{Examples of saturating subspaces}\label{S:4}
  
  \subsection{Linearised equation}\label{S:4.1}

  Let $s\ge 0$ be an arbitrary number,~$\II\subset\Z^d$ be a finite  set containing the zero vector, and~$\HH(\II)$ be defined  by~\eqref{1.10}.~We denote by~$\HH_k(\II)$,~$k\in \N\cup\{\ty\}$   the subspaces~$\HH_k$ given by relations~\eqref{1.6} and~\eqref{1.7} for~$\HH=\HH(\II)$.  
   In this section, we prove the following result. 
      \begin{proposition}\label{P:4.1} The subspace $\HH_\ty(\II)$ is danse in $H^s$
        if and only if~$\II$ is a generator.
\end{proposition}
\begin{proof}    Assume that~$\II$ is a generator. From the identities
\begin{align*}
2\cos\lag l,x\rag \cos\lag m,x\rag&= \cos\lag l-m,x\rag+\cos\lag l+m,x\rag , \\ 
2\sin\lag l,x\rag \sin\lag m,x\rag&= \cos\lag l-m,x\rag-\cos\lag l+m,x\rag ,\\	
2\sin\lag l,x\rag \cos\lag m,x\rag&= \sin\lag l-m,x\rag+\sin\lag l+m,x\rag , \\ 
2\cos\lag l,x\rag \sin\lag m,x\rag&=   \sin\lag m-l,x\rag+\sin\lag l+m,x\rag 
\end{align*} it follows that  if $l,m\in \Z_*^d$ are   such that 
$$
 \cos\lag m,x\rag,\, \sin\lag m,x\rag\in\HH(\II)\,\,\text{ and} \,\, \cos\lag l,x\rag,\, \sin\lag l,x\rag\in  \HH_k(\II) 
$$for some $k\ge0$,
  then  
$$
  \cos\lag m\pm l,x\rag,\, \sin\lag m\pm l,x\rag \in \HH_{k+1}(\II).
$$
This implies that
$$
 \cos\lag l,x\rag,\, \sin\lag   l,x\rag \in \HH_\ty(\II) \quad \text{for any $l\in \Z^d$},
$$
so the space $\HH_\ty(\II)$ is danse in $H^s$ for any $s\ge0$.

If $\II$ is not a generator, we argue as in the proof of Theorem~\ref{T:3.9}. Let~$l$ be any vector in~$  \Z^d_*\setminus \tilde \II$. From the above trigonometric identities it follows that~$\HH_\ty (\II)\subset \HH(\tilde \II)$.~Thus,  the functions     $\cos\lag l,x\rag$ and $\sin\lag l,x\rag$    are orthogonal to the subspace~$\HH_\ty(\II)$ in~$H^s$.~Hence,      $\HH_\ty(\II)$ is 
not dense in~$H^s$. \end{proof}

  \subsection{Nonlinear equation}\label{S:4.2}
  
  We start this section with a proof of   Lemma~\ref{L:3.3}.
  \begin{proof}[Proof of Lemma~\ref{L:3.3}]
 Let us denote by~$\GG$ the subspace on the right-hand side of~\eqref{3.9}.~Clearly, we have $\HH_j'\subset \GG$. To prove the inverse inclusion, 
 we take any vectors~$\zeta_l\in \HH_{j-1}'$, $l=1, \ldots, 2p+1$ and 
   consider the function $F:\R^{2p+1}\to \HH_j'$ defined~by
 $$
 F(x_1, \ldots, x_{2p+1})=B\left(\sum_{l=1}^{2p+1} x_l \zeta_l\right).
 $$
 As the subspace $\HH_j'$ is invariant under complex conjugation, we have $\Re \zeta \in \HH_j'$ for any $\zeta\in \HH_j'$.~Let us choose $\zeta_l\in \HH_{j-1}'$, $l=1, \ldots, 2p+1$ to be real vectors. Then
  $$
 F(x_1, \ldots, x_{2p+1})= \left(\sum_{l=1}^{2p+1} x_l \zeta_l\right)^{2p+1}.
 $$
  As the subspace $\HH_j'$ is closed, we have  
$$
\frac{\p^{2p+1}}{\p_{x_1}\ldots \p_{x_{2p+1}}} F(0,\ldots,0)=(2p+1)!\,\zeta_1\cdot\ldots\cdot \zeta_{2p+1}\in \HH_j'.
$$  By linearity, this implies that 
  $\zeta_1\cdot\ldots\cdot \zeta_{2p+1}\in \HH_j'$ for any vectors $\zeta_l\in \HH_{j-1}'$, $l=1, \ldots, 2p+1$.
 \end{proof}
 Let $s\ge0$ be arbitrary,~$\II\subset\Z^d$ be a finite  set containing   zero, and~$\HH_k'(\II)$, $k\in \N\cup\{\ty\}$ be the subspaces  defined by~\eqref{3.7} and~\eqref{3.8} for~$\HH=\HH(\II)$.
       \begin{proposition}\label{P:4.2}
       The~subspace $\HH_\ty'(\II)$   is   dense in $H^s$   if and only if~$\II$ is a generator.  
  \end{proposition}
\begin{proof}Let $\II$ be a   generator.~From Lemma~\ref{L:3.3}   and  the assumption that~$1\in \HH(\II)$ it follows that $\HH_j(\II)\subset \HH_j'(\II)$ for any $j\ge1$.~Thus, we have $\HH_\ty(\II)\subset \HH_\ty'(\II)$, so that~$\HH_\ty'(\II)$ is danse in $H^s$, by 
Proposition~\ref{P:4.1}.~The other assertion is proved as in Proposition~\ref{P:4.1}, by noticing that~$\HH_\ty' (\II)\subset \HH(\tilde \II)$.
\end{proof}

\def\cprime{$'$} \def\cprime{$'$}
  \def\polhk#1{\setbox0=\hbox{#1}{\ooalign{\hidewidth
  \lower1.5ex\hbox{`}\hidewidth\crcr\unhbox0}}}
  \def\polhk#1{\setbox0=\hbox{#1}{\ooalign{\hidewidth
  \lower1.5ex\hbox{`}\hidewidth\crcr\unhbox0}}}
  \def\polhk#1{\setbox0=\hbox{#1}{\ooalign{\hidewidth
  \lower1.5ex\hbox{`}\hidewidth\crcr\unhbox0}}} \def\cprime{$'$}
  \def\polhk#1{\setbox0=\hbox{#1}{\ooalign{\hidewidth
  \lower1.5ex\hbox{`}\hidewidth\crcr\unhbox0}}} \def\cprime{$'$}
  \def\cprime{$'$} \def\cprime{$'$} \def\cprime{$'$}

\end{document}